\documentclass
{article}
\usepackage[T2A]{fontenc}
\usepackage[cp1251]{inputenc}
\usepackage[english,russian]{babel}
\usepackage[tbtags]{amsmath}
\usepackage{amsfonts,amssymb,mathrsfs,amscd}
\usepackage{wasysym}
\usepackage{verbatim}
\usepackage{eucal}
\usepackage{upgreek}
\usepackage{graphicx}
\usepackage{enumerate}
\let\No=\textnumero

\usepackage[hyper]{msb-a}
\JournalName{}
\numberwithin{equation}{section}

\theoremstyle{plain}

\newtheorem{maintheorem}{Теорема-критерий}
\newtheorem{maincorol}{Следствие}
\newtheorem{lemma}{Лемма}

\theoremstyle{definition}
\newtheorem{proof}{Доказательство}
\newtheorem{remark}{Замечание}

\renewcommand{\leq}{\leqslant} 
\renewcommand{\geq}{\geqslant}
\newcommand{\RR}{\mathbb{R}} 
\newcommand{\CC}{\mathbb{C}} 
\newcommand{\NN}{\mathbb{N}} 

\DeclareMathOperator{\supp}{{\sf supp}}
\DeclareMathOperator{\dd}{\,{\mathrm d\!}}
\DeclareMathOperator{\sbh}{{\sf sbh}}  

\DeclareMathOperator{\rad}{{\text{\tiny \rm rd}}}

\DeclareMathOperator{\dsh}{{\sf dsh}}

\begin{document} 
\title{Интегралы  от разностей субгармонических  функций.~II. Один критерий}
	
\author[B.\,N.~Khabibullin]{Б.\,Н.~Хабибуллин}
\address{Башкирский государственный университет}
\email{khabib-bulat@mail.ru}

\date{28.06.2021}
\udk{517.547.2 : 517.574}

 \maketitle

\begin{fulltext}

\begin{abstract}  Получен критерий возможности оценки интегралов от произвольной разности субгармонических функций по мере  через их характеристику Неванлинны исключительно в терминах самой меры или потенциала от неё.  Эти результаты новые и для мероморфных функций.     

Библиография: 6  названий 

Ключевые слова: мероморфная функция, разность субгармонических функций, характеристика Неванлинны,  мера Рисса,  мера конечной энергии, полярное множество





\end{abstract}

\markright{Интегралы  от разностей субгармонических  функций. II \dots}

\section{Введение}
Используются обозначения, определения и соглашения первой части \cite{Kha21II_1} нашей работы, но по возможности мы их напоминаем. 

\subsection{Формулировка следствия для мероморфных функций}
 Сформулируем на\-иболее существенную часть результатов критерия для разности субгармонических функций из настоящей второй части нашей работы   только  для мероморфных функций в окрестности  замкнутого круга  $\overline D(R)\subset \CC$ с центром в нуле радиуса $R>0$, т.е. {\it на $\overline D(R)$,} которые тем более   справедливы для мероморфных функций на  $\CC$.  

Через  $\overline D_z(r):=z+\overline D(r)$ обозначаем замкнутые  круги с центром $z\in \CC$ радиуса $r\in \RR^+$, а для меры Бореля $\mu$ на $\CC$ полагаем 
\begin{equation}\label{muzr}
\mu^{\rad}_z(r):=\mu\bigl(\overline D_z(r)\bigr), \quad {\sf N}^{\mu}_z(r):=\int_0^r\frac{\mu^{\rad}_z(t)}{t}\dd t, \quad r\in \RR^+.
\end{equation}

\begin{maincorol}\label{mthm}
Пусть $0<r\in \RR^+$,  $\mu$ ---  мера Бореля на   $\overline D(r)\subset \CC$. 
Тогда следующие два утверждения эквивалентны: 
\begin{enumerate}[{\rm I.}]
\item\label{UIm} Существуют $r_0>0$ и $R>r$, для которых $\sup\limits_{z\in \overline D(R)}{\sf N}_z^{\mu}(r_0)<+\infty$.

\item\label{UIIm} При  любом $R>r$   для каждой мероморфной функции  $f\neq \infty$ на замкнутом круге 
$\overline D(R)$  функция $\ln^+|f|$  $\mu$-суммируема  и 
\begin{equation*}
\hspace{-2mm}\int_{\overline D(r)} \ln^+|f|\dd \mu \leq  
5\frac{R+r}{R-r}  \bigl(T(R,f)-N(r,f)\bigr)
\biggl(\mu\bigl(\overline D(r)\bigr)+\sup_{z\in \overline D(R)} {\sf N}_z^{\mu}(r)\biggr)<+\infty.
\end{equation*}
\end{enumerate}
\end{maincorol}
Здесь $T(R,f)$ --- классическая характеристика Неванлинны функции $f$, а в обозначении 
$n(r,f)$ для числа полюсов функции $f$ в $\overline D(r)$ 
$$
N(r,f):=\int_{0}^{r}\frac{n(t,f)-n(0,f)}{t}\dd t+n(0,f)\ln r 
$$  --- возрастающая  функция и  в первом неравенстве  утверждения \ref{UIIm}  в $N( r, f)$  можно заменить $r$ на любое    $r'\in [0,r]$ с сохранением конечности правой части при $r'\neq 0$. Кроме того,  в случае $r\geq 1$ или  $n(0,f)=0$, т.е. при $f(0)\in \CC$, имеем  $N( r, f)\geq 0$, и в этом неравенстве в таких случаях можно и убрать $-N( r, f)$ в скобках. Ввиду декларируемой в утверждении \ref{UIIm} конечности промежуточной части неравенства, утверждение  \ref{UIm} очевидным образом следует из утверждения \ref{UIIm} при выборе $r_0:=r$. Вывод утверждения \ref{UIIm} из \ref{UIm} --- следствие теоремы-критерия о пяти эквивалентных утверждениях, связанных с  интегралами от разности субгармонических функций по мере в шаре. Сформулирована и доказана эта теорема в \S~\ref{MTh}.

\subsection{Краткая сводка  некоторых определений}\label{Nch}

Через $B_x(r)$, $\overline B_x(r)$, а также $\partial \overline B_x(r)$ обозначаем  соответственно {\it открытый\/} и  {\it замкнутый шар,\/} а также  {\it  сферу радиуса $r\in \RR^+$  с центром\/} $x\in \RR^{\tt d}$, ${\tt d}\geq 2$.  В случае $x=0\in \RR^+$ нижний индекс $0$ не пишем. 

{\it Площади поверхностей единичных сфер\/} $\partial \overline  B(1)$ в $\RR^{\tt d}$ обозначаем через 
\begin{equation}\label{{kK}s}
s_{\tt d-1}=\frac{2\pi^{\tt d/2}}{\Gamma (\tt d/2)},  
\quad s_{\tt 1}=2\pi, \; s_{\tt 2}=4\pi,\;  s_{\tt 3}=\pi^2, \dots  . 
\end{equation}
Для  {\it поверхностной меры площади\/} $\sigma_{\tt d-1}^r$ на  $ \partial \overline B(r)\subset \RR^{\tt d}$ и интегрируемой по  $\sigma_{\tt d-1}^r$ функции $U\colon \partial \overline B(r)\to \overline \RR$  её {\it среднее по сфере $ \partial \overline B(r)$} обозначается как  
\begin{equation}\label{CU}
{\sf C}_U(r):=\frac{1}{s_{\tt d-1}r^{\tt d-1}}\int_{\partial \overline B(r)}U\dd \sigma_{\tt d-1}^r.
\end{equation}
Меры Бореля $\mu$, заданные  на борелевских подмножествах в $\RR^{\tt d}$, часто рассматриваем как продолженные на всё $\RR^{\tt d}$  и
\begin{equation}\label{muyr} 
\mu_y^{\rad}(t)\underset{t\in \RR^+}{:=}\mu\bigl(\overline B_y(t) \bigr)\in \overline \RR^+, 
\end{equation}
--- {\it радиальная считающая функция меры $\mu$.} 
В случае центра $y=0$ нижний индекс $0$ также  не используем. 
Неоднократно будет использоваться связанное с размерностью ${\tt d}\in \NN$ число 
\begin{equation}\label{kd0}
\widehat{\tt d}:=\max\{{\tt 1,d-2}\}=1+({\tt d-3})^+\in \NN.  
\end{equation}
Для меры Бореля $\mu$ на $\overline B(R)\subset \RR^{\tt d}$ 
\begin{equation}\label{Ntt}
{\sf N}_{\mu}(r,R):={\widehat{\tt d}}\int_r^R \frac{\mu^{\rad}(t)}{t^{\tt d-1}}\dd t\in \overline \RR^+ \quad\text{при  $ 0\leq r<R<R_0\in \overline \RR^+$}
\end{equation}
--- её {\it  радиальная разностная проинтегрированная считающая функция,} а 
 \begin{equation}\label{Nymu}
{\sf N}_y^{\mu}(r)\overset{\eqref{muzr}}{:=}\widehat{\tt d}\int_0^r \frac{\mu_y^{\rad}(t)}{t^{\tt d-1}}\dd t\in \overline \RR^{\tt d},
\quad r\in  \overline \RR^+, 
\end{equation}
--- её {\it  радиальная проинтегрированная считающая функция с центром $y\in \RR^{\tt d}$.}

Для $\delta$-суб\-г\-а\-р\-м\-о\-н\-и\-ч\-е\-с\-ких функций $U\not\equiv \pm\infty $ на {\it замкнутом шаре\/} 
$\overline B(R)\subset \RR^{\tt d}$ её {\it разностная  характеристика Неванлинны} определяется как 
\begin{equation}\label{{rT}N}
{\boldsymbol T}_U(r,R)\overset{\eqref{CU},\eqref{Ntt}}{=}{\sf C}_{U^+}(R)+{\sf N}_{\varDelta_U^-}(r,R),
\quad 0\leq r<R\in \RR^+,
\end{equation}
где мера $\varDelta_U^-\geq 0$ --- {\it нижняя вариация заряда Рисса\/} функции $U$.
В частности, для мероморфной функции $f\neq \infty$ на круге $\overline D(R)\subset \CC$ 
\begin{equation}\label{csTTNN}
T(R, f)-N(r, f)={\boldsymbol  T}_{\ln|f|}(r,R),
\quad 0< r<R\in {\mathbb{R}}^+.
\end{equation}

Для {\it связного\/} подмножества  $S\subset \RR^{\tt d}$ через $\sbh(S)$ обозначаем множество всех функций,  субгармонических на какой-либо области $D_u$, содержащей  $S$, или, другими словами, {\it на $S$.}
Подмножество $\sbh_*(S)\subset \sbh(S)$ состоит из всех функций, не равных тождественно  $-\infty$ в области $D_u\supset S$, для которых пишем $u\not\equiv -\infty$ на $S$.
 Класс $\dsh(S)$ {\it $\delta$-субгармонических функций на $S$} состоит
функций, заданных  как разность $U=u-v$ пары  функций $u, v\in \sbh(S)$, исключая разность функций, тождественно равных $-\infty$, а $\dsh_*(S)$  --- разности $U=u-v$ функций $u, v\in \sbh_*(S)$, для которых пишем 
$U\not\equiv \pm\infty$ на $S$. 

Расширенная  числовая функция 
\begin{equation}\label{kKd-2}
\Bbbk_{\tt d-2} \colon  t\underset{0<t\in \RR^+}{\longmapsto} \begin{cases}
\ln t  &\text{\it  при ${\tt d=2}$},\\
 -\dfrac{1}{t^{\tt d-2}} &\text{\it при ${\tt d>2}$,} 
\end{cases} 
\quad \Bbbk (0):=-\infty \in \overline \RR,
\end{equation}
очевидно,   строго  возрастающая и  непрерывная на $ \RR^+$. 

 Разностная характеристика Неванлинны ${\boldsymbol T}_U$ $\delta$-суб\-г\-а\-р\-м\-о\-н\-и\-ч\-е\-с\-к\-ой функции $U\not\equiv \pm \infty$ на шаре с центром в нуле положительна,
 возрастающая и выпуклая относительно $\Bbbk_{\tt d-2}$ по второй большей переменной, а также 
убывающая  и вогнутая относительно $\Bbbk_{\tt d-2}$ по  первой переменной. 

\section{Основной результат}\label{MTh}

\begin{maintheorem}\label{mth}
Пусть $0<r\in \RR^+$,  $\mu$ ---  мера Бореля на $\overline B(r)\subset \RR^{\tt d}$. 
Следующие пять   утверждений эквивалентны: 
\begin{enumerate}[{\rm I.}]
\item\label{UI}  В обозначении \eqref{Nymu} выполнено соотношение 
\begin{equation}\label{k0CR}
\sup_{y\in \overline B(R)}{\sf N}_y^{\mu}(r_0)<+\infty \quad\text{для некоторых  $r_0>0$ и $R>r$.}
\end{equation} 
\item\label{UII} Для любого $R>r$   каждая функция $U\in {\sf dsh}_*\bigl(\overline B(R)\bigr)$ $\mu$-суммируема  и 
\begin{subequations}\label{UR}
\begin{align}
\hspace{-3mm}\int_{\overline B(r)} U^+\dd \mu &\leq  A_{\tt d}(r,R){\boldsymbol  T}_U( r, R)\biggl(\mu^{\rad}(r)\max\{1, r^{2-\tt d}\} +\sup_{y\in \overline B(r)}{\sf N}_y^{\mu}(r)\biggr),
\tag{\ref{UR}T}\label{{UR}T}\\
\intertext{где правая  часть неравенства конечна и  
}
 A_{\tt d}(r,R)&:=5\max\bigl\{1, {\tt d}-2\bigr\}\Bigl(\frac{R+r}{R-r}\Bigr)^{\tt d-1}\max\Bigl\{1, (R-r)^{\tt d-2}\Bigr\},
\tag{\ref{UR}A}\label{{UR}A}
\end{align}
\end{subequations}
а первый аргумент $r$ в ${\boldsymbol  T}_U( r, R)$ можно заменить на любое    $r'\in [0,r]$.
\item\label{UIII} Существуют число $R>r$,  для которого  все  функции 
$U\in \dsh_*\bigl(\overline B(R)\bigr)$ $\mu$-интегрируемы, и такое число  $T>0$, что 
\begin{equation}\label{RM}
\sup\biggl\{\int_{\overline B(r)} U^+\dd \mu \biggm| {\boldsymbol  T}_U( r, R)\leq T,
\; U\in \dsh_*\bigl(\overline B(R)\bigr)  \biggr\}<+\infty.
\end{equation}

\item\label{UIV} Мера $\mu$ конечна, а её $\mu$-потенциал, обозначаемый и определяемый как  
\begin{equation}\label{Ukd}
{\sf pt}_{\mu}\colon x\underset{x\in \RR^{\tt d}}{\longmapsto}   \int_{\RR^{\tt d}} \Bbbk_{\tt d-2}\bigl(|y-x|\bigr)\dd \mu(y),
\end{equation}
ограничен снизу на носителе $\supp \mu\subset \overline B(r)$.

\item\label{UV} Мера $\mu$ конечна и 
\begin{equation}\label{k0CRr}
\sup_{y\in \supp \mu}{\sf N}_y^{\mu}(r_0)<+\infty \quad\text{для некоторого  $r_0>0$.}
\end{equation}
\end{enumerate}

\end{maintheorem}

\begin{proof}
При $\mu= 0$ всё очевидно, поэтому далее $\mu\neq 0$. 

Символ импликации $\overset{?}{\Longrightarrow}$ с вопросительным знаком над ним далее показывает, что ниже  доказывается именно эта  импликация.

\paragraph{{\rm \ref{UI}}$\overset{?}{\Longrightarrow}${\rm \ref{UII}}}
 Неоднократно будет использована следующая элементарная  
\begin{lemma}[{\cite[предложение 2.2]{KhaShm19}}]\label{prok}
Пусть   $0<r\in \RR^+$ и  $h\colon (0,r]\to \RR^+$ --- возрастающая функция. 
Если   сходится  интеграл Римана 
\begin{equation}\label{{se:fcc}0}  
\int_{0}^{r}\frac{h(t)}{t^{\tt d-1}}\dd t<+\infty,
\end{equation} 
то  существуют пределы  
\begin{equation}\label{es:00l} 
h(0):=\lim_{0<t\to 0} h(t)=0,
\quad \lim_{0<t\to 0} h(t)\Bbbk_{\tt d-2}( t)=0, 
\end{equation} 
а также сходится интеграл Римана\,--\,Стилтьеса 
\begin{equation}\label{ex:limf0} 
\int_0^{r}\Bbbk_{\tt d-2}(t) \dd h(t)>-\infty. 
\end{equation} 
Обратно, если   выполнено \eqref{ex:limf0},  то  существуют пределы 
\begin{equation}\label{es:00l0} 
h(0):=\lim_{0<t\to 0} h(t)\in \RR^+,
\quad \lim_{0<t\to 0}\bigl(h(t)-h(0)\bigr)\Bbbk_{\tt d-2}( t)=0 
\end{equation} 
и сходится интеграл 
\begin{equation}\label{f-f0}
\int_{0}^{r}\frac{h(t)-h(0)}{t^{\tt d-1}}\dd t<+\infty. 
\end{equation} 
При любом из  условий  \eqref{{se:fcc}0}  или  \eqref{ex:limf0}  имеет место  равенство 
\begin{equation}\label{in:poch0}
{\widehat{\tt d}}\int_{0}^{r}\frac{h(t)-h(0)}{t^{\tt d-1}}\dd t=
\int_{0}^{r}\bigl(\Bbbk_{\tt d-2}(r)-\Bbbk_{\tt d-2}(t)\bigr)\dd h(t),
\end{equation}
а  при условии  \eqref{{se:fcc}0}  и $r_0\in (0,r]$ выполняется  неравенство
\begin{equation}\label{hHmd}
{\widehat{\tt d}}\int_0^r\frac{h(t)}{t^{\tt d-1}}\dd t\leq h(r)
\bigl(\Bbbk_{\tt d-2}(r)-\Bbbk_{\tt d-2}(r_0)\bigr)+{\widehat{\tt d}}\int_0^{r_0}\frac{h(t)}{t^{\tt d-1}}\dd t.
\end{equation} 
\end{lemma}

Также несколько раз будет использована 

\begin{lemma}\label{lemmu} При условии  \eqref{k0CR}  мера $\mu$ конечна,
\begin{equation}\label{NyR}
\sup_{y\in \RR^{\tt d}}{\sf N}_y^{\mu}(t)
= \sup_{y\in \overline B(r)}{\sf N}_y^{\mu}(t) <+\infty \quad\text{при  любом   $t\in \RR^+\setminus 0$,}
\end{equation} 
потенциал ${\sf pt}_{\mu}$ ограничен снизу на всём $\RR^{\tt d}$, 
а для любого борелевского полярного множества его $\mu$-мера равна нулю. 
\end{lemma}
\begin{proof}[леммы \ref{lemmu}] Из  \eqref{k0CR} по определению \eqref{Nymu}   $\mu$-меры шаровых слоев  $B_y(r_0)\setminus B_y\bigl(\min\{R-r,r_0\}/2\bigr)$ равномерно ограничены при $y$, пробегающем $\overline B(R)$. Такие шаровые слои покрывают {\it компакт\/} $\overline B(r)$, откуда мера $\mu$ конечна.  Из  неравенства \eqref{hHmd} леммы  \ref{prok} с функция $h:=\mu_y^{\rad}$ получаем 
\begin{equation*}
{\sf N}_y^{\mu}(t)\leq \mu(\RR^{\tt d}) 
\bigl(\Bbbk_{\tt d-2}(t)-\Bbbk_{\tt d-2}(r_0)\bigr)+{\sf N}_y^{\mu}(r_0)
\quad\text{для любых  $y\in \RR^{\tt d}$ и $t\geq r_0$}. 
\end{equation*}
Применение к обеим частям этого неравенства точной верхний грани 
по всем $y\in \overline B(r)$ даёт последнее строгое неравенство $<+\infty$ из \eqref{NyR}. 
Для любой точки  $y$ вне шара $\overline B(r)$ по  неравенству треугольника  пересечение $\overline B_y(t)\cap \overline B(r)$ содержится в шаре $\overline B_{y'}(t)$ с центром  $y':=ry/|y| \in \overline B(r)$, откуда  по определению 
\eqref{Nymu} ${\sf N}_y^{\mu}(t)\leq {\sf N}_{y'}^{\mu}(t)$ при $|y|>r$ и получаем первое равенство в  \eqref{NyR}.

Из \eqref{k0CR}  по  лемме  \ref{prok} с $h:=\mu_y^{\rad}$ из существования пределов \eqref{es:00l}
\begin{equation}\label{muh0?}
\lim_{0<t\to 0} \mu_y^{\rad}(t)=0,
\quad \lim_{0<t\to 0} \mu_y^{\rad}(t)\Bbbk_{\tt d-2}( t)=0 
\end{equation}
и равенства \eqref{in:poch0} для любого $R\in \RR^+\setminus 0$  по определению 
\eqref{Nymu} получаем 
\begin{multline*}
{\sf N}_y^{\mu}(R)\overset{\eqref{Nymu}}{=}\widehat{\tt d}\int_0^R \frac{\mu_y^{\rad}(t)}{t^{\tt d-1}}\dd t
\overset{\eqref{in:poch0}}{=} \int_{0}^{R}\bigl(\Bbbk_{\tt d-2}(R)-\Bbbk_{\tt d-2}(t)\bigr)\dd \mu_y^{\rad}(t)
\\=\mu^{\rad}(R)\Bbbk_{\tt d-2}(R)- \int_{\overline B_y(R)} \Bbbk_{\tt d-2}\bigl(|x-y|\bigr)\dd \mu(x)
\quad\text{для всех $y\in \RR^{\tt d}$.}
\end{multline*}
Но при выборе $R\geq 2r$ шар $\overline B_y(R)$ содержит в себе шар $\overline B(r)$ и тогда 
\begin{equation*}
{\sf N}_y^{\mu}(R)\overset{\eqref{Ukd}}{=}
\mu^{\rad}(R)\Bbbk_{\tt d-2}(R) -{\sf pt}_{\mu}(y)\quad\text{для всех $y\in \RR^{\tt d}$ и $R\geq 2r$}.
\end{equation*}
Отсюда для конечной меры $\mu$, удовлетворяющей  \eqref{Ukd},  получаем 
\begin{equation*}
\inf_{y\in \RR^{\tt d}} {\sf pt}_{\mu}(y)>-\infty, 
\end{equation*}
и потенциал ${\sf pt}_{\mu}$ ограничен снизу на всём $\RR^{\tt d}$. 
Тогда  {\it интеграл энергии\/} 
\begin{equation*}
I[\mu]:=\int_{\RR^{\tt d}}{\sf pt}_{\mu}\dd \mu
\end{equation*}
конечен и  $\mu$ --- {\it мера конечной энергии,\/} а для таких мер мера любого борелевского полярного множества равна нулю \cite[теорема 3.2.3]{Rans}, \cite[теорема II.2]{Helms}. 
\end{proof}

Пусть $u\not\equiv -\infty$ и $v\not\equiv -\infty$  --- пара субгармонических функций на $\overline B(R)$, представляющих $U=u-v$. Значения этой разности определены и конечны в каждой точке $x\in \overline B(R)\setminus E$  вне 
 {\it полярного борелевского множества\/}
\begin{equation}\label{E}
E=\bigl\{x\in B(R)\bigm| u(x)=-\infty\bigr\}\bigcup \bigl\{x\in B(R)\bigm| v(x)=-\infty\bigr\}
\end{equation}  
  На $\overline B(r)\setminus E$ функция $U^+$ всюду определена как положительная часть разности  $u-v$  полунепрерывных сверху функций $u$ и $v$ со значениями в $\RR$ и является измеримой по сужению меры Бореля $\mu$ на  $\overline B(r)\setminus E$, а значит и интегрируемой  по этому сужению. 
В то же время из заключительной части леммы \ref{lemmu} следует, что $\mu(E)=0$, откуда получаем   $\mu$-интегрируемость функции $U^+$ на всём шаре
$\overline B(r)$. Её $\mu$-суммируемость будет следовать из неравенства \eqref{{UR}T}, поскольку правая часть в этом неравенстве по соотношениям \eqref{NyR} из леммы \ref{lemmu} конечна. 

Переходим к доказательству неравенства \eqref{UR} для функции $U=u-v$.

Применяя формулу Пуассона\,--\,Йенсена  в шаре  $B(R)$ \cite[(3.7.3)]{HK} к $u(x)$ и $v(x)$
в каждой точке  $x\in \overline B (r)$, а затем вычитая одно равенство из другого
 в каждой точке   $x\in \overline B(R)$, лежащей вне $E$ из \eqref{E}, получаем равенство 
\begin{multline}\label{Ux}
U(x)=\frac{1}{{\sf s}_{\tt d-1}}\int_{\partial \overline B(R)} \frac{R^2-|x|^2}{R|y-x|^{\tt d}}U(y)\dd \sigma_{\tt d-1}^R (y)\\
-\int_{B(R)}\Biggl( \Bbbk_{\tt d-2}\biggl(\Bigl|\frac{R}{|y|}y-\frac{|y|}{R}x\Bigr| \biggr)-\Bbbk_{\tt d-2}\bigl(|y-x|\bigr)\Biggr)\dd \varDelta_U(y),\;  x\in \overline B(r)\setminus E. 
\end{multline}

Здесь для {\it положительного ядра Пуассона\/} \cite[1.5.4]{HK} имеем 
\begin{equation*}
\frac{1}{{\sf s}_{\tt d-1}} \frac{R^2-|x|^2}{R|y-x|^{\tt d}}\leq 
\frac{1}{{\sf s}_{\tt d-1}} \frac{R+r}{R(R-r)^{\tt d-1}}
\quad\text{при всех $y\in \partial \overline B(R)$ и $x\in \overline B(r)$},
\end{equation*}
а для {\it положительной функции Грина\/}  \cite[теорема 1.10]{HK} ---
\begin{multline*}
\Bbbk_{\tt d-2}\biggl(\Bigl|\frac{R}{|y|}y-\frac{|y|}{R}x\Bigr| \biggr)-\Bbbk_{\tt d-2}\bigl(|y-x|\bigr)
\\
\leq \Bbbk_{\tt d-2}(R+r)-\Bbbk_{\tt d-2}\bigl(|y-x|\bigr)
\quad \text{при всех $y\in B(R)$ и  $x\in \overline B(r)$}.
\end{multline*}
Таким образом, из \eqref{Ux} следует, что при всех  $x\overset{\eqref{E}}{\in} \overline B(r)\setminus E$ 
\begin{multline*}
U^+(x)\leq \frac{R^{\tt d-2}(R+r)}{(R-r)^{\tt d-1}}{\sf C}_{U^+}(R)
+\int_{B(R)}\Bigl(\Bbbk_{\tt d-2}(R+r)-\Bbbk_{\tt d-2}\bigl(|y-x|\bigr)\Bigr) \dd\varDelta_U^-(y).
\end{multline*}
Теперь ввиду $\mu$-интегрируемости функции $U^+$ можем интегрировать  по мере $\mu$ это неравенство и использовать теорему Фубини  о повторных интегралах: 
\begin{multline}\label{RTD}
\int_{\overline B(r)}U^+\dd \mu
\leq \int_{\overline B(r)}\frac{R^{\tt d-2}(R+r)}{(R-r)^{\tt d-1}}{\sf C}_{U^+}(R)\dd \mu(x)
\\+\int_{\overline B(r)}\int_{B(R)}\Bigl(\Bbbk_{\tt d-2}(R+r)-\Bbbk\bigl(|y-x|\bigr)\Bigr) \dd\varDelta_U^-(y)\dd \mu (x)\\
=\frac{R^{\tt d-2}(R+r)}{(R-r)^{\tt d-1}}{\sf C}_{U^+}(R)\mu^{\rad}(r)
\\+\int_{B(R)}\int_{\overline B(r)}\Bigl(\Bbbk_{\tt d-2}(R+r)-\Bbbk\bigl(|y-x|\bigr)\Bigr) \dd \mu (x)\dd\varDelta_U^-(y)
\\
\overset{\eqref{muyr}}{=}
\frac{R^{\tt d-2}(R+r)}{(R-r)^{\tt d-1}}{\sf C}_{U^+}(R)\mu^{\rad}(r)\\
+\int_{B(R)}\int_0^{R+r}\Bigl(\Bbbk_{\tt d-2}(R+r)-\Bbbk_{\tt d-2}(t)\Bigr) \dd \mu_y^{\rad} (t)\dd\varDelta_U^-(y)\\
\overset{\eqref{{rT}N}}{\leq} \frac{R^{\tt d-2}(R+r)}{(R-r)^{\tt d-1}}{\boldsymbol T}_{U}(r,R)\mu^{\rad}(r)
\\
+(\varDelta_U^-)^{\rad}(R)
\sup_{y\in B(R)}
\int_0^{R+r}\Bigl(\Bbbk_{\tt d-2}(R+r)-\Bbbk_{\tt d-2}(t)\Bigr) \dd \mu_y^{\rad} (t),
\end{multline}
где при последнем переходе использовано и неравенство 
 ${\sf C}_{U^+}(R)\overset{\eqref{{rT}N}}{\leq} {\boldsymbol{T}_U}(r, R)$, вытекающее из 
 определения разностной характеристики Неванлинны \eqref{{rT}N}.
Для последнего  интеграла Римана\,--\,Стилтьеса,  используя равенство \eqref{in:poch0} 
леммы  \ref{prok} с $h:=\mu_y^{\rad}$ и определение \eqref{Nymu} с очевидным 
 $\mu_y^{\rad} (r)\leq \mu^{\rad} (r)$, получаем 
\begin{multline}\label{Bbbkh}
\int_0^{R+r}\Bigl(\Bbbk_{\tt d-2}(R+r)-\Bbbk_{\tt d-2}(t)\Bigr) \dd \mu_y^{\rad} (t)
\overset{\eqref{in:poch0}}{=}\widehat{\tt d}\int_0^{R+r}\frac{\mu_y^{\rad} (t)}{t^{\tt d-1}}\dd t
\\
\overset{\eqref{hHmd}}{\leq}\Bigl(\Bbbk_{\tt d-2}(R+r)-\Bbbk_{\tt d-2}(r)\Bigr) \mu_y^{\rad} (r)+
\widehat{\tt d}\int_0^{r}\frac{\mu_y^{\rad} (t)}{t^{\tt d-1}}\dd t\\
\leq \Bigl(\Bbbk_{\tt d-2}(R+r)-\Bbbk_{\tt d-2}(r)\Bigr) \mu^{\rad} (r)+
{\sf N}_y^{\mu}(r). 
\end{multline}
Все установленные выше соотношения остаются справедливыми и для любого промежуточного $R_*>r$, не превышающего $R$:
\begin{equation}\label{R*}
r<R_*<R , \quad \overline B(r)\subset B(R_*)\subset \overline B(R).  
\end{equation}
Таким образом,  \eqref{RTD} с учётом  \eqref{Bbbkh}  можем записать как 
\begin{subequations}\label{R**}
\begin{align}
&\int_{\overline B(r)}U^+\dd \mu
\leq \frac{R_*^{\tt d-2}(R_*+r)}{(R_*-r)^{\tt d-1}}{\boldsymbol{T}_U}(r,R)\mu^{\rad}(r)
\tag{\ref{R**}T}\label{R**T}
\\
&
\hspace{-1mm}
+(\varDelta_U^-)^{\rad}( R_*)\biggl(
\mu^{\rad}(r)\bigl(\Bbbk_{\tt d-2}(R_*+r)-\Bbbk_{\tt d-2}(r)\bigr)+
\sup_{y\in \overline B(r)}{\sf N}_y^{\mu}(r) 
\biggr).
\tag{\ref{R**}$\varDelta$}\label{R**vD}
\end{align}
\end{subequations}
где использовано равенство 
$$\sup\limits_{y\in B(R)}{\sf N}_y^{\mu}(r)=\sup\limits _{y\in \overline B(r)}{\sf N}_y^{\mu}(r),$$
вытекающее из равенства в \eqref{NyR} в лемме \ref{lemmu}.

\begin{lemma}\label{lemd} Пусть   $\varDelta$ --- мера Бореля на $\overline B(R)\subset \RR^{\tt d}$ и\/ $0<R_*<R$.
 Тогда
\begin{equation}\label{1}
\varDelta^{\rad}(R_*)
\leq 
\frac{{\sf N}_{\varDelta} (R_*,R)}{\Bbbk_{\tt d-2}(R)-\Bbbk_{\tt d-2}(R_*)} 
\quad\text{при любых ${\tt d}\geq 2$.}
\end{equation}
\end{lemma} 
\begin{proof}[леммы \ref{lemd}] В силу возрастания считающей функции $\varDelta^{\rad}$
\begin{equation*}
\varDelta^{\rad}(R_*)\leq \int_{R_*}^{R}\frac{\varDelta^{\rad}(t)}{t^{\tt d-1}}\dd t\biggm/
\int_{R_*}^{R} \frac{\dd t}{t^{\tt d-1}}
\overset{\eqref{Ntt}}{=}
\frac{{\sf N}_{\varDelta} (R_*,R)}{\Bbbk_{\tt d-2}(R)-\Bbbk_{\tt d-2}(R_*)}.
\end{equation*}
\end{proof}
По лемме \ref{lemd} применительно к  $(\varDelta_U^-)^{\rad}(R_*)$ с учётом неравенств 
$$
{\sf N}_{\varDelta_U^-}(R_*,R)\overset{\eqref{Ntt}}{\leq} {\sf N}_{\varDelta_U^-}(r,R)\overset{\eqref{{rT}N}}{\leq} {\boldsymbol{T}}_U(r,R)
$$
для части  \eqref{R**vD}  неравенства \eqref{R**} получаем 
\begin{multline}\label{kkk}
\eqref{R**vD} \leq {\boldsymbol{T}}_U(r,R)  \frac{1}{\Bbbk_{\tt d-2}(R)-\Bbbk_{\tt d-2}(R_*)}
\\
\times\biggl(\mu^{\rad}(r)\Bigl(\Bbbk_{\tt d-2}(R_*+r)-\Bbbk_{\tt d-2}(r)\Bigr)
+\sup_{y\in \overline B(r)}{\sf N}_y^{\mu}(r)\biggr)
\end{multline}

Рассмотрим сначала отдельно 
\paragraph{Случай ${\tt d=2}$} 
В \eqref{R*} выберем $R_*:=\sqrt{rR} $ как среднее геометрическое чисел $r$ и $R$.
Для части \eqref{R**vD} неравенства \eqref{R**} из  \eqref{kkk}  получаем  
\begin{multline*}
\eqref{R**vD}\leq {\boldsymbol{T}}_U(r,R)  
\frac{1}{\ln \sqrt{R/r}}\biggl(
\mu^{\rad}(r)\ln \frac{\sqrt{Rr}+r}{r}+
\sup_{z\in \overline D(r)}{\sf N}_z^{\mu}(r) 
\biggr)\\
={\boldsymbol{T}}_U(r,R)  
\biggl(\mu^{\rad}(r)+\mu^{\rad}(r)\frac{\ln\bigl(1+\sqrt{r/R}\bigr)}{\ln \sqrt{R/r}}+
\frac{2}{\ln (R/r)}\sup_{z\in \overline D(r)}{\sf N}_z^{\mu}(r) 
\biggr)\\
\leq 
{\boldsymbol{T}}_U(r,R)  
\biggl(\mu^{\rad}(r)+\mu^{\rad}(r)\frac{2\ln 2}{\ln (R/r)}+
\frac{2}{\ln (R/r)}
\sup_{z\in \overline D(r)}{\sf N}_z^{\mu}(r) 
\biggr).
\end{multline*} 
Отсюда, используя неравенства 
$$
\frac{1}{\ln (R/r)}=1\biggm/ \int_r^R\frac{\dd t}{t}\leq \frac{R}{R-r}\geq 1 \quad\text{при $0<r<R$},
$$
 можем продолжить предыдущие неравенства как 
\begin{multline*}
\eqref{R**vD} \leq {\boldsymbol{T}}_U(r,R)  
\biggl(\frac{R}{R-r}\mu^{\rad}(r) +\frac{R}{R-r}\mu^{\rad}(r) 2\ln  2+2\frac{R}{R-r}\sup_{z\in \overline D(r)}{\sf N}_z^{\mu}(r) 
\biggr)\\
\leq \frac{R}{R-r}{\boldsymbol{T}}_U(r,R)  
\biggl((1+2\ln 2)\mu^{\rad}(r) +2\sup_{z\in \overline D(r)}{\sf N}_z^{\mu}(r)\biggr),
\end{multline*}
откуда согласно   \eqref{R**}  следует 
\begin{multline*}
\int_{\overline B(r)}U^+\dd \mu
\leq {\boldsymbol{T}_U}(r,R) \biggl(
\frac{\sqrt R+\sqrt r}{\sqrt R-\sqrt r}\mu^{\rad}(r)
\\
+\frac{R}{R-r}
\Bigl((1+2\ln 2)\mu^{\rad}(r)+2\sup_{z\in D(R)}{\sf N}_z^{\mu}(r) \Bigr)\biggr)\\
\leq {\boldsymbol{T}_U}(r,R) \Biggl(
\biggl(2\frac{R+r}{ R- r}
+\frac{(1+2\ln 2)R}{R-r}\biggr)\mu^{\rad}(r)+\frac{2R}{R-r}\sup_{z\in D(R)}{\sf N}_z^{\mu}(r) \Bigr)\Biggr)
\\
\leq 5\frac{R+r}{R-r}\biggl(
\mu^{\rad}(r)+\sup_{z\in D(R)}{\sf N}_z^{\mu}(r) \biggr),
\end{multline*}
что устанавливает требуемое неравенство \eqref{{UR}T} с выписанным в \eqref{{UR}A} 
 $$
A_{\tt 2}(r,R):=5\frac{R+r}{R-r}\quad \text{ при ${\tt d=2}$.}
$$

\paragraph{Случай ${\tt d}>2$} 
Из очевидного $$
\Bbbk_{\tt d-2}(R_*+r)-\Bbbk_{\tt d-2}(r)\leq \frac{1}{r^{\tt d-2}}
$$ для  ${\tt d}>2$
и элементарного 
\begin{equation*}
{\Bbbk_{\tt d-2}(R)-\Bbbk_{\tt d-2}(R_*)}=\frac{1}{{\tt d}-2}\int_{R_*}^R\frac{\dd t}{t^{\tt d-1}}\geq 
\frac{1}{{\tt d}-2}\frac{R-R_*}{R^{{\tt d}-1}} 
\end{equation*} 
при выборе $R_*:=\frac12(R+r) $ как среднего арифметического чисел $r$ и $R$
для части \eqref{R**vD} неравенства \eqref{R**} из  \eqref{kkk}  получаем 
\begin{equation*}\label{d>2}
\eqref{R**vD}\leq {\boldsymbol{T}}_U(r,R)  
\frac{2({\tt d}-2)R^{\tt d-1}}{R-r}\biggl(
\mu^{\rad}(r)\frac{1}{r^{\tt d-2}}+
\sup_{z\in \overline B(r)}{\sf N}_z^{\mu}(r) \biggr)\quad \text{при ${\tt d>2}$},
\end{equation*}
откуда согласно   \eqref{R**}  следует 
\begin{multline*}
\int_{\overline B(r)}U^+\dd \mu
\leq 2\Bigl(\frac{R+r}{R-r}\Bigr)^{\tt d-1}
{\boldsymbol{T}_U}(r,R)\mu^{\rad}(r)
\\
+2({\tt d}-2){\boldsymbol{T}}_U(r,R)  
\frac{R^{\tt d-1}}{R-r}\biggl(
\mu^{\rad}(r)\frac{1}{r^{\tt d-2}}+
\sup_{z\in \overline B(r)}{\sf N}_z^{\mu}(r) \biggr)
\\
\leq 2\Bigl(\frac{R+r}{R-r}\Bigr)^{\tt d-1}{\boldsymbol{T}}_U(r,R)  
\biggl(\Bigl(1+({\tt d}-2)(R-r)^{\tt d-2}\Bigr)\mu^{\rad}(r)\max\{1, r^{2-\tt d}\}
\\+({\tt d}-2)(R-r)^{\tt d-2}
\sup_{z\in \overline B(r)}{\sf N}_z^{\mu}(r) \biggr)\\
\leq 
2\Bigl(\frac{R+r}{R-r}\Bigr)^{\tt d-1} {\boldsymbol{T}}_U(r,R) \cdot 
2({\tt d}-2)\max\Bigl\{1, (R-r)^{{\tt d}-2}\Bigr\}\\\biggl(\mu^{\rad}(r)\max\{1, r^{2-\tt d}\}+
\sup_{z\in \overline B(r)}{\sf N}_z^{\mu}(r) \biggr),
\end{multline*}
что даёт  требуемое неравенство \eqref{{UR}T} с $A_{\tt d}(r,R)$ из \eqref{{UR}A}  при ${\tt d}>2$.
Правая часть неравенства \eqref{{UR}T} конечна, если учесть \eqref{NyR} из леммы \ref{lemmu}. 

\paragraph{{\rm \ref{UII}}$\overset{?}{\Longrightarrow}${\rm \ref{UIII}}}
Для какого-нибудь фиксированного числа $R>r$ рассмотрим неравенство \eqref{UR} для гармонической функции $U\equiv 1$, для которой, очевидно, ${\boldsymbol{T}}_1(r,R)=1$.  Из конечности правой части в  \eqref{{UR}T} при  выборе $U\equiv 1$ следует существование числа $M>0$, не зависящего от $U\in \dsh_*\bigl(\overline B(R)\bigr)$, для которого  
\begin{equation*}
A_{\tt d}(r,R)\biggl(\mu^{\rad}(r)\max\{1, r^{2-\tt d}\}+\sup_{y\in \overline B(R)}
{\sf N}_y^{\mu}(r)\biggr)\leq M<+\infty.
\end{equation*}
Отсюда для любой функции $U\in \dsh_*\bigl(\overline B(R)\bigr)$
при  $ {\boldsymbol  T}_U( r, R)\leq T$ вновь из неравенства \eqref{{UR}T} получаем 
$$
\int_{\overline B(r)} U^+\dd \mu  \leq MT<+\infty, 
$$
где произведение $MT$ не зависит от $U$, что доказывает   \eqref{RM}.  

\paragraph{{\rm \ref{UIII}}$\overset{?}{\Longrightarrow}${\rm \ref{UIV}}}
Соотношение \eqref{RM} справедливо для любого $T>0$, поскольку умножение функции $U\in \dsh_*\bigl(\overline B(R)\bigr)$ на положительное число умножает на него и разностную характеристику  Неванлинны. 
Кроме того, из соотношения \eqref{RM} утверждения \ref{UIII} гармоническая функция, тождественно равная единице, $\mu$-суммируема, откуда  мера $\mu$ конечна.  

Для  $y\in \RR^{\tt d}$ 
рассмотрим {\it субгармонические\/} функции 
\begin{equation}\label{kd}
k_y\colon x\overset{\eqref{{kK}s}}{\underset{x \in \RR^{\tt d}}{\longmapsto}}
\Bbbk_{\tt d-2} \bigl(|x-y|\bigr)\quad\text{с мерой Рисса $\varDelta_{k_y}=\updelta_y$}
 \end{equation}
где $\updelta_y$ ---  {\it мера Дирака в точке $y$,\/} т.е. вероятностная  мера с носителем  $\supp \updelta_y=\{y\}$. Из   формулы Пуассона\,--\,Йенсена \eqref{Ux}, применённой к значениям функций  $U:=k_y$ в нуле сразу следует равенство
\begin{equation}\label{Cky}
{\sf C}_{k_y}(R)= \Bbbk_{\tt d-2}(R)\quad \text{при всех $y\in B(R)$, $0<R\in \RR^+$}. 
\end{equation}

Для  существующего значения $R>r$ из утверждения   \ref{UIII} рассмотрим семейство {\it супергармонических функций\/} $\{K_y\}_{y\in \RR^{\tt d}}$ на  всем $\RR^{\tt d}$, определённых как 
\begin{equation}\label{k-p}
K_y\colon x\underset{x\in \RR^{\tt d}}{\longmapsto} \Bbbk_{\tt d-2}(R+r)-\Bbbk_{\tt d-2}\bigl(|x-y|\bigr),
\end{equation}  
Если  $y$ лежит на $\overline B(r)$, 
то  функции $K_y$ {\it положительны на $\overline B(R)$} и 
\begin{equation}\label{KC}
{\sf C}_{K_y^+}(R)={{\sf C}_{K_y}}(R)\overset{\eqref{Cky}}{=}\Bbbk_{\tt d-2}(R+r)-\Bbbk_{\tt d-2}(R), 
\quad  y\in \overline B(r).
\end{equation}
При этом заряд Рисса супергармонической  функции $K_y$  противоположен мере  Дирака в точке $y$ и 
справедливы равенства 
\begin{equation*}
\varDelta_{K_y}=-\updelta_y=-\varDelta_{K_y}^- , 
\quad \varDelta_{K_y}^-(t)=
\begin{cases}
0\text{ при $t\in \bigl[0,|y|\bigr)$},\\
1\text{ при $t\in \bigl[|y|, +\infty\bigr)$}.
\end{cases}
\end{equation*} 
Отсюда при $|y|\leq r$ получаем 
\begin{equation*}
{\sf N}_{\varDelta_{K_y}^-}(r, R)=\widehat{\tt d}\int_{r}^R\frac{1}{t^{\tt d-1}}\dd t
\overset{}{=}\Bbbk_{\tt d-2}(R)-\Bbbk_{\tt d-2}(r), \quad  y\in \overline B(r),
\end{equation*}
что  с \eqref{KC} по определению разностной характеристики Неванлинны даёт 
\begin{equation}\label{TTk}
{\boldsymbol{T}_{K_y}}(r,R)={\sf C}_{K_y^+}(R)+ {\sf N}_{\varDelta_{K_y}^-}(r, R)=
\Bbbk_{\tt d-2}(R+r)-\Bbbk_{\tt d-2}(r),
\end{equation}
где правая часть строго положительна и не зависит от $y\in \overline B(r)$. Выберем число $T$ равным правой части 
\eqref{TTk}. Тогда по  соотношению \eqref{RM} утверждения \ref{UIII} с функциями $U\overset{\eqref{k-p}}{:=}K_y$
существует число $C$, для которого 
\begin{equation*}
\sup_{y\in \overline B(r)} \int_{\overline B(r)}K_y(x)\dd \mu(x)\leq C<+\infty,
\end{equation*}
а из явного вида \eqref{k-p} функции $K_y$ получаем 
$$
\inf_{y\in \overline B(r)} \int_{\overline B(r)}\Bbbk_{\tt d-2}\bigl(|x-y|\bigr)\dd \mu(x)\geq \Bbbk_{\tt d-2}(R+r)\mu^{\rad}(r)-C > -\infty.
$$
 
\paragraph{{\rm \ref{UIV}}$\overset{?}{\Longrightarrow}${\rm \ref{UI}}}
Выберем произвольное число $R>r$. 
Из ограниченности снизу потенциала ${\sf pt}_{\mu}$ на $\supp \mu\subset \overline B(r)$ 
следует его ограниченность снизу всюду на $\RR^{\tt d}$ \cite[теорема 3.1.4]{Rans}, \cite[теорема 1.10]{Landkof}, что
из представления  потенциала ${\sf pt}_{\mu}$ интегралом Римана\,--\,Стилтьеса даёт
\begin{equation}\label{kRmu}
\inf_{y\in \overline B(R)}\int_0^{+\infty}\Bbbk_{\tt d-2}(t) \dd \mu_y^{\rad}(t)>-\infty. 
\end{equation}
Но при $y\in \overline B(R)$ шар $B_y(R+r)$ включает в себя шар $\overline B(r)$, следовательно,
$\mu_y^{\rad}(t)\equiv \mu^{\rad}(r)$ при $t\geq R+r$, а верхний предел интегрирования в  \eqref{kRmu}
можно заменить на $R+r$, откуда 
\begin{equation*}
\sup_{y\in \overline B(R)}\int_0^{R+r}\Bigl( \Bbbk_{\tt d-2}(R+r)- \Bbbk_{\tt d-2}(t) \Bigr)\dd \mu_y^{\rad}(t)<+\infty. 
\end{equation*} 
и по лемме \ref{prok} с $h:=\mu_y^{\rad}$  в условиях \eqref{ex:limf0} из равенств 
\eqref{es:00l0} и \eqref{in:poch0} получаем 
\begin{equation*}
\sup_{y\in \overline B(R)} \widehat{\tt d}\int_0^{R+r}\frac{\mu_y^{\rad}(t)}{t^{\tt d-1}} \dd t <+\infty.
\end{equation*}
Это 
при   $r_0:=R+r$ даёт требуемое в утверждении \ref{UI} соотношение \eqref{k0CR}.

\paragraph{{\rm \ref{UI}}$\overset{?}{\Longrightarrow}${\rm \ref{UV}}}
 Конечность меры $\mu$ --- из леммы \ref{lemmu}, а \eqref{k0CRr} --- частный случай \eqref{k0CR}. 

\paragraph{{\rm \ref{UV}}$\overset{?}{\Longrightarrow}${\rm \ref{UIV}}}
Положим $R:=r_0+2r>2r>r$. 
В силу конечности меры $\mu$ 
\begin{equation}\label{k0CRr0}
\sup_{y\in \supp \mu}{\sf N}_y^{\mu}(R) \leq \sup_{y\in \supp \mu}{\sf N}_y^{\mu}(r_0)+
\mu^{\rad}(r)\int_{r_0}^R\frac{\dd t}{t^{\tt d-1}} <+\infty.
\end{equation}
Отсюда, в частности, следует условие \eqref{{se:fcc}0} леммы \ref{prok}
  для $h:=\mu_y^{\rad}$  при каждом  $y\in \supp \mu$. Следовательно, по равенствам \eqref{es:00l} и \eqref{in:poch0} имеем 
\begin{equation*}
{\sf N}_y^{\mu}(R)={\widehat{\tt d}}\int_{0}^{R}\frac{\mu_y^{\rad}(t)}{t^{\tt d-1}}\dd t\overset{\eqref{in:poch0}}{=}
\int_{0}^{R}\bigl(\Bbbk_{\tt d-2}(R)-\Bbbk_{\tt d-2}(t)\bigr)\dd \mu_y^{\rad}(t)
\end{equation*}
и, как следствие,   
\begin{equation}\label{kinf}
\inf_{y\in \supp \mu}\int_{0}^{R}\Bbbk_{\tt d-2}(t)\dd \mu_y^{\rad}(t)\geq
-\bigl| \Bbbk_{\tt d-2}(R) \bigr| \mu^{\rad}(r) -\sup_{y\in \supp \mu}{\sf N}_y^{\mu}(R).
\end{equation}
Но при всех $y\in \supp \mu\subset \overline B(r)$ имеем  $\overline B_y(R)\supset \overline B(r)$, откуда 
$\mu_y(t)\equiv \mu_y(R)$ при всех $t\geq R$. Это даёт при всех $y\in \supp \mu$ равенства 
$$
\int_{0}^{R}\Bbbk_{\tt d-2}(t)\dd \mu_y^{\rad}(t)=\int_{0}^{+\infty}\Bbbk_{\tt d-2}(t)\dd \mu_y^{\rad}(t)
=\int_{\RR^{\tt d}}\Bbbk_{\tt d-2}\bigl(|x-y|\bigr)\dd \mu(x)={\sf pt}_{\mu}(y),
$$
что в сочетании с \eqref{kinf} и \eqref{k0CRr0} обеспечивает ограниченность снизу потенциала 
${\sf pt}_{\mu}$ на носителе $\supp \mu$ и выполнение утверждения  \ref{UIV}.
\end{proof}
\begin{remark}
Импликация \ref{UIm}$\Longrightarrow$\ref{UIIm} основного следствия из введения --- очевидное следствие 
импликации \ref{UI}$\Longrightarrow$\ref{UII} основной теоремы. 
\end{remark}

\end{fulltext}

\end{document}